\documentclass[12pt]{article}
\pdfoutput=1

\usepackage[utf8]{inputenc}
\usepackage{geometry}
\usepackage{amsmath,amsfonts,amsthm,amscd,upref,amstext, amssymb}
\usepackage{verbatim}
\usepackage[shortlabels]{enumitem}
\usepackage{xcolor}

\usepackage[
backend=biber,
style=alphabetic,
]{biblatex}
\title{A bibLaTeX example}
\usepackage{url}
\newtheorem{theorem}{Theorem}[section]
\newtheorem{corollary}[theorem]{Corollary}
\newtheorem{lemma}[theorem]{Lemma}
\newtheorem{claim}[theorem]{Claim}
\newtheorem{subclaim}[theorem]{Subclaim}
\newtheorem{fact}[theorem]{Fact}
\newtheorem{definition}[theorem]{Definition}
\newtheorem{conjecture}[theorem]{Conjecture}

\newcommand\I{\mathcal{I}}
\newcommand\J{\mathcal{J}}

\renewcommand\H{\mathcal{H}}

\newcommand\T{\mathcal{T}}
\newcommand\R{\mathbb{R}}

\newcommand\Gam{\bf{\Gamma}}

\usepackage{fancyhdr}
\usepackage{lipsum} 

\title{Unreachability of $\bf{\Gamma_{2n+1,m}}$}
\author{Derek Levinson}
\date{\today}

\begin{document}
\maketitle

\abstract{\noindent
We find bounds for the maximal length of a sequence of distinct $\bf{\Gamma_{2n+1,m}}$-sets under $AD$ and show there is no sequence of distinct $\bf{\Gamma_{2n+1}}$-sets of length $\bf{\delta^1_{2n+3}}$. As a special case, there is no sequence of distinct $\bf{\Gamma_{1,m}}$-sets of length $\aleph_{m+2}$. These are the optimal results for the pointclasses $\bf{\Gamma_{2n+1}}$ and $\bf{\Gamma_{1,m}}$.

\section{Introduction}
\label{Intro}

We will assume $ZF + AD + DC$ throughout this paper. A goal of descriptive set theory since its early days was to identify, for each natural (boldface) pointclass $\bf{\Gamma}$, the supremum of the lengths of prewellorderings in $\bf{\Gamma}$. In the 1980s, the work of Moschovakis, Kechris, Martin, Jackson, and others culminated in a computation of $\bf{\delta}^1_n$ = $sup\{|\leq^*| : \,\leq^* \text{ is a prewellordering in } \bf{\Delta^1_n}\}$ for each $n\in \omega$ (see \cite{mosdst} and \cite{jackson_2011}). A related question is the supremum of the lengths of sequences of distinct sets from $\bf{\Gamma}$. This has also been resolved for the pointclasses in the projective hierarchy, albeit much more recently (see \cite{hra}). But there remained many interesting pointclasses strictly between $\bf{\Sigma^1_{2n}}$ and $\bf{\Delta^1_{2n+1}}$ in the Wadge hierarchy for which the problem was unsolved. In this paper we study the lengths of sequences of distinct sets from some of these pointclasses:

\begin{definition}
    For  $n\in\omega\backslash\{0\}$, $m\in\omega$, $k\in\omega\backslash\{0\}$, and $x_0\in\R$, the pointclass $\Gamma_{n,m,k}(x_0)$ is the set of $A\subseteq\R$ such that there exists a $\Sigma_k$-formula $\psi$ satisfying
    \begin{align*}
        (\forall y\in \R)\, y\in A \longleftrightarrow M_{n-1}\models \psi(y,x_0,s_m),
    \end{align*}
    where $s_m = (u_0,...,u_{m-1})$ is the sequence of the first $m$ uniform indiscernibles. Define $\bf{\Gamma_{n,m,k}}$$ = \bigcup_{x_0\in\R} \Gamma_{n,m,k}(x_0)$, $\bf{\Gamma_{n,m}}$ = $\bigcup_{k\in\omega} \bf{\Gamma_{n,m,k}}$, and $\bf{\Gamma_n}$ = $\bigcup_{m\in \omega} \bf{\Gamma_{n,m}}$.
\end{definition}

$\bf{\Gamma_n}$ is the envelope of $\bf{\Pi^1_{n+1}}$ --- the smallest boldface pointclass which contains scales for universal $\bf{\Sigma^1_{n+1}}$ and $\bf{\Pi^1_{n+1}}$ sets.  In particular, if $n$ is odd the optimal scale on a universal $\bf{\Pi^1_{n+1}}$-set is cofinal in $\bf{\Gamma_n}$. Similarly, if $n$ is even the optimal scale on a universal $\bf{\Sigma^1_{n+1}}$-set is cofinal in $\bf{\Gamma_n}$.

There is an alternative characterization of $\bf{\Gamma_n}$ in purely descriptive set-theoretic terms using the difference hierarchy. We say $A\subseteq \R$ is $\alpha-\bf{\Pi^1_1}$ if there is a sequence $\langle A_\beta : \beta < \alpha\rangle\subset \bf{\Pi^1_1}$ such that $x\in A \longleftrightarrow min\{\beta : \beta = \alpha \vee x \notin A_\beta\}$ is odd. Then $\bf{\Gamma_n}$ $ = \bigcup_{m<\omega} \Game^n (\omega\cdot m - \bf{\Pi^1_1})$.

\begin{definition}
    Let $a_{n,m} = sup\{|\leq^*|:\,\leq^*\text{ is a prewellordering in }\bf{\Gamma_{n,m}}\}$.
\end{definition}

\begin{definition}
    Suppose $P$ is a complete iterate of $M_n(x)$ and $k>0$. Define
    \begin{align*}
        &\gamma^P_{m,k}= sup(Hull^P_k(u_0,...,u_{m-1})\cap \delta_P),\\
        &\gamma^n_{m,k,x,\infty} = \pi_{M_n(x),\infty}(\gamma^{M_n(x)}_{m,k}),\\
        &b_{n,m,k} = sup_{x\in\R} \gamma^n_{m,k,x,\infty},\text{ and}\\
        &b_{n,m} = sup_{k\in\omega} b_{n,m,k},
    \end{align*}
    where $\delta_P$ is the least Woodin cardinal of $P$ and $\pi_{M_n(x),\infty}$ is the direct limit embedding from $M_n(x)$ into the direct limit of all countable iterates of $M_n(x)$ by its unique iteration strategy.
\end{definition}

Hjorth computed $a_{1,m}$ and $b_{1,m}$ for every $m\in \omega$ (see \cite{bddness_lemma}):

\begin{theorem}[Hjorth]
\label{hjorth pwo thm}
    $a_{1,m} = b_{1,m} = u_{m+1}$.\footnote{We are writing $u_m$ for the $m+1$st-uniform indiscernible. This is off by one from the notation of \cite{bddness_lemma}.}
\end{theorem}

In fact, Hjorth's arguments also show $b_{1,m,k} = u_{m+1}$ for any $m\in \omega$ and $k\in\omega\backslash\{0\}$. Under $AD$, $u_m = \aleph_{m+1}$. In \cite{on_the_pwo}, Sargsyan extends the analysis of \cite{bddness_lemma} to the rest of the pointclasses in the projective hierarchy.

The value of $b_{2n+1,m}$ is unknown for $n>0$, but it is known that $b_{2n+1,m} < \kappa^1_{2n+3}$ for all $m\in\omega$ and $sup_{m<\omega} b_{2n+1,m} = \kappa^1_{2n+3}$, where $\kappa^1_{2n+3}$ is the Suslin cardinal satisfying $S(\kappa^1_{2n+3}) = \bf{\Sigma^1_{2n+3}}$.
The proof of 
Lemma 6.1 of \cite{on_the_pwo} gives:\footnote{A cautionary note: Sargsyan defines $b_{m,n}$ differently in \cite{on_the_pwo} than we do above. But his main results still hold with our definition.}
\begin{theorem}
\label{b is cardinal}
    $b_{2n+1,m,k}$ is a cardinal.
\end{theorem}

We now turn from prewellorderings to the lengths of sequences of distinct sets from a pointclass.

\begin{definition}
    For a pointclass $\bf{\Gamma}$ and a cardinal $\kappa$, we say $\kappa$ is $\bf{\Gamma}$-reachable if there is a sequence of distinct $\bf{\Gamma}$-sets of length $\kappa$. Otherwise, we say $\kappa$ is $\bf{\Gamma}$-unreachable.
\end{definition}

The minimal $\bf{\Gamma}$-unreachable cardinal is at least $\bf{\delta_\Gamma}$. In \cite{Kechris}, Kechris proved $\bf{\delta^1_{2n+2}}$ is $\bf{\Delta^1_{2n+1}}$-unreachable and conjectured that it is also $\bf{\Sigma^1_{2n+2}}$-unreachable. Jackson made progress towards this conjecture in \cite{Jackson}:

\begin{theorem}[Jackson]
    $\bf{\delta^1_{2n+2}}$ is $\bf{\Delta^1_{2n+2}}$-unreachable.
\end{theorem}

Hjorth and Sargsyan resolved Kechris's conjecture:

\begin{theorem}[Hjorth]
\label{hjorth unreachability thm}
    $\bf{\delta^1_2}$ is $\bf{\Sigma^1_2}$-unreachable.
\end{theorem}

\begin{theorem}[Sargsyan]
\label{grigor unreachability thm}
    $\bf{\delta^1_{2n+2}}$ is $\bf{\Sigma^1_{2n+2}}$-unreachable.
\end{theorem}

This is the optimal result for the pointclass $\bf{\Sigma^1_{2n+2}}$. Hjorth proved Theorem \ref{hjorth unreachability thm} in \cite{twoapp}. Only in the last few years did Sargysan discover his generalization to the rest of the projective hierarchy (see \cite{hra}). Neeman, Sargsyan, and the author gave an alternate proof of Theorem \ref{grigor unreachability thm} which also applies to any inductive-like pointclass in $L(\R)$ (see \cite{phdthesis}).

\begin{theorem}[Levinson, Neeman, Sargsyan]
\label{ind-like thm}
    Assume $ZF +AD + DC + V=L(\R)$. If $\bf{\Gamma}$ is inductive-like, then $\bf{\delta_\Gamma^+}$ is $\bf{\Gamma}$-unreachable.
\end{theorem}

Locating the least $\bf{\Gamma_{2n+2,m}}$-unreachable cardinal follows trivially from Theorem \ref{grigor unreachability thm} and prior work. By a result of Moschovakis, there is a $\bf{\Pi^1_{2n+3}}$ prewellordering of length $\bf{\delta^1_{2n+3}}$. In particular, $\bf{\delta^1_{2n+3}}$ is $\bf{\Delta^1_{2n+3}}$-reachable. But $\bf{\delta^1_{2n+4}}$ (the successor of $\bf{\delta^1_{2n+3}}$) is $\bf{\Sigma^1_{2n+4}}$-unreachable. Since $\bf{\Delta^1_{2n+3}} \subset \bf{\Gamma^1_{2n+2,m}} \subset \bf{\Sigma^1_{2n+4}}$, the least $\bf{\Gamma_{2n+2,m}}$-unreachable cardinal must be $\bf{\delta^1_{2n+4}}$.

On the other hand, many cardinals lie between $\bf{\delta^1_{2n+2}}$ (the minimal $\bf{\Sigma^1_{2n+2}}$-unreachable cardinal) and $\bf{\delta^1_{2n+3}}$ (the greatest $\bf{\Delta^1_{2n+3}}$-reachable cardinal) and almost nothing was known about unreachability for pointclasses between $\bf{\Sigma^1_{2n+2}}$ and $\bf{\Delta^1_{2n+3}}$. We address this gap in the main theorem of this paper:

\begin{theorem}
\label{main thm}
    $b_{2n+1,m,k+1}$ is $\bf{\Gamma_{2n+1,m,k}}$-unreachable.
\end{theorem}

Since $sup_{m<\omega} b_{2n+1,m} = \kappa^1_{2n+3}$ and $(\kappa^1_{2n+3})^+ = \bf{\delta^1_{2n+3}}$, the theorem implies 

\begin{corollary}
\label{main cor}
    $\bf{\delta^1_{2n+3}}$ is $\bf{\Gamma_{2n+1}}$-unreachable.
\end{corollary}

This is the optimal result for the pointclass $\bf{\Gamma_{2n+1}}$. Theorem \ref{main thm} also gives optimal bounds for unreachability of  $\bf{\Gamma_{1,m}}$. As above, the minimal $\bf{\Sigma^1_2}$ and $\bf{\Delta^1_3}$-unreachable cardinals were known to be $\bf{\delta^1_2} = \aleph_2$ and $\bf{\delta^1_4} = \aleph_{\omega+2}$, respectively. It was open whether there exist any $m\in [3,\omega+1]$ such that $\aleph_m$ is the minimal $\bf{\Gamma}$-unreachable cardinal for some pointclass $\bf{\Gamma}$. This cannot hold for $m=\omega$ because $\aleph_\omega$ has cofinality $\omega$. But Theorem \ref{main thm} implies it holds for every other $m\in[3,\omega+1]$:

\begin{corollary}
\label{cor for n=0}
    $\aleph_{m+2}$ is the least cardinal which is $\bf{\Gamma_{1,m}}$-unreachable. $\aleph_{\omega+1}$ is the least cardinal which is $\bf{\Gamma_1}$-unreachable.
\end{corollary}

The corollary is immediate from \ref{main thm} and \ref{hjorth pwo thm}. Corollary \ref{cor for n=0} strengthens Theorem \ref{hjorth unreachability thm}, since $\bf{\Sigma^1_{2}} \subset \Gamma_{1,0}$ and $\aleph_2 = \bf{\delta^1_2}$. As another corollary of Theorem \ref{main thm}, we get a bound on the lengths of prewellorderings in $\bf{\Gamma_{2n+1,m}}$ which is at least as good as the bound given in \cite{on_the_pwo}:

\begin{corollary}
\label{cor on pwo}
    $a_{2n+1,m}\leq b_{2n+1,m}$
\end{corollary}

The proofs of Theorems \ref{hjorth unreachability thm}, \ref{grigor unreachability thm}, \ref{ind-like thm}, and \ref{main thm} all use the technique of directed systems of mice. \cite{hra} bounds the length of a sequence of $\bf{\Sigma^1_{2n+2}}$ sets by coding each member of the sequence by a single condition in the extender algebra of the direct limit of all countable iterates of $M_{2n+1}$ below its least cardinal which is strong up to its least Woodin. \cite{phdthesis} uses a different coding to analyze the inductive-like cases --- each set of reals is coded by a set of conditions contained in the extender algebra of the direct limit of a suitable mouse below, but possibly cofinal in, its least cardinal which is strong up to its Woodin cardinal. Our proof of \ref{main thm} will use the same directed system as \cite{hra}, but the coding is more similar to that of \cite{phdthesis} --- each $\bf{\Gamma_{2n+1,m,k}}$-set will be coded by (roughly) a set of conditions in the extender algebra of the direct limit of $M_{2n+1}$ below its least cardinal which is strong up to its least Woodin relative to the theory of the first $m$ uniform indiscernibles.

\section{Background and Notation}

We will assume the reader is familiar with the basic theory of mice as developed in \cite{ooimt}. Here a mouse will refer to a premouse with an $\omega_1+1$-iteration strategy. In this section we mention a few concepts and bits of notation which are less standard. A more detailed exposition of most of the material in this section appears in Section 1 of \cite{hra}.

\subsubsection*{Woodin's Extender Algebra}

For a model $M$ with at least one Woodin cardinal, let $\delta_M$ denote the least Woodin cardinal of $M$. We write $Ea_M$ for Woodin's extender algebra in $M$ at $\delta_M$. $Ea_M$ is a $\delta_M$-c.c. complete Boolean algebra. (see Theorem 7.14 of \cite{ooimt}). We will write $ea$ for a generic for $Ea_M$. When considering the product forcing $Ea_M\times Ea_M$, we will write $ea_l$ and $ea_r$ for the left and right generics, respectively. $ea_r$ will often code a triple, which shall be written $(ea_r^1,ea_r^2,ea_r^3)$.

\subsubsection*{Directed Systems of Mice}

We say a premouse $N$ is a complete iterate of a mouse $M$ if there is an iteration tree $\T$ on $M$ with last model $N$ such that $\T$ is according to the iteration strategy for $M$ and $\T$ has no drops on the branch to $N$. In this case, we write $\pi_{M,N}$ for the iteration embedding from $M$ to $N$.

Suppose $M$ is a mouse with at least one Woodin cardinal and let $\I(M)$ be the set of all complete iterates of $M$ by countable iteration trees below $\delta_M$. If $N\in \I(M)$, we write $\pi_{N,\infty}$ for the direct limit embedding of $N$ into the direct limit of $\I(M)$.

Fix $n<\omega$ and let $\I = \I(M_{2n+1})$. Consider some $x_0\in\R$ such that a real recursive in $x_0$ codes $M_{2n+1}$ and let $M\in \I(M_{2n+1}(x_0))$. We will want to use the approximation of $\I$ inside $M$ developed in \cite{hra}. Let $\nu$ be the least inaccessible cardinal of $M$ above $\delta_M$ and let
\begin{align*}
    \I^M = \I \cap M|\nu.
\end{align*}
$\H^M$ will denote the direct limit of $\I$. Since $\I^M$ is countable, $\H^M\in \I$. 

If $N\in \I^M$ and $\zeta < \delta_N$, the map $\pi_{N,\H^M}\upharpoonright \zeta$ is definable in $M$. More is true:

\begin{fact}
\label{direct limit approx definable}
    If $y$ is $Ea_M$-generic, $N \in M[y]\cap \I$, and $\zeta<\delta_N$, then $\H^M$ is a complete iterate of $N$ and $\pi_{N,\H^M}\upharpoonright\zeta$ is definable in $M[y]$ (uniformly in $M$, $y$, and $N$). 
\end{fact}

Fact \ref{direct limit approx definable} is not difficult to prove from the analysis of \cite{on_the_pwo}. A proof of an analogous fact is given in \cite{phdthesis} for a directed system on a different mouse.

\subsubsection*{The StrLe Construction}

We require a way of constructing an $M_{2n+1}(x)$-like mouse inside an $M_{2n+1}(y)$-like mouse when $x\leq_T y$. Our method for this is the ``strong $L[\vec{E}]$-construction.'' Suppose $M$ is a $y$-premouse with at least one Woodin cardinal and $x\in M\cap \R$. Let $N$ be the output of the fully-backgrounded Mitchell-Steel construction over $x$ performed in $M|\delta_M$. Let $StrLe[M,x]$ be the result of Steel's $S$-construction\footnote{Referred to as the $P$-construction in \cite{self-iter}.} in $M$ over $N$. The next two facts can be found in \cite{hra}.

\begin{fact}
\label{generic for strle}
    Suppose $M$, $y$, and $x$ are as above and $z\in M\cap\R$. Let $S = StrLe[M,x]$. Then $z$ is $Ea_S$-generic over $S$.
\end{fact}

\begin{fact}
\label{strle in directed system}
    Suppose $x,y\in\R$, a real recursive in $y$ codes a complete iterate $\bar{M}$ of $M_n(x)$ by a countable iteration tree below $\delta_{M_n(x)}$, and $S = StrLe[M_n(y),x]$. Then $S$ is a complete iterate of $\bar{M}$ by a countable iteration tree below $\delta_{\bar{M}}$.
\end{fact}

Consider a premouse $M$ with at least $n$ Woodin cardinals and $x\in M \cap \R$. Let $\bar{S}$ be the first level of $StrLe[M,x]$ such that $L[\bar{S}]$ has $n$ Woodin cardinals. For any formula $\psi$, real $x\in M$, and finite sequence of ordinals $s$,  we write $M\models ``M_n(x) \models \psi(x,s)$'' to mean $L[\bar{S}]\models \psi(x,s)$. It is clear this is definable in $M$ from parameters in $\{x\}\cup s$. Moreover,

\begin{fact}
\label{strle formula definable}
    For $M$ as above, $k > 1$, and $\psi$ a $\Sigma_k$-formula, $M\models ``M_n(x) \models \psi(x,s)$'' is $\Sigma_k$-definable over $M$ from parameters in $\{x\}\cup s$.
\end{fact}

Let us justify this notation. Suppose $M$ is a $y$-mouse with at least $n$ Woodin cardinals and no extenders indexed at or above $u_0$. Let $x\in M \cap \R$ and $\bar{S}$ be as above. Then $L[\bar{S}]$ is an $M_n(x)$-like mouse (with no extenders indexed at or above $u_0$). In particular, the first order theory of the uniform indiscernibles in $L[\bar{S}]$ is the same as the first order theory of the uniform indiscernibles in $M_n(x)$. So if $s = (u_0,..., u_m)$, then $M\models ``M_n(x) \models \psi(x,s)$'' if and only if $M_n(x) \models \psi(x,s)$.

\section{The Proof}

In this section we prove Theorem \ref{main thm}.

Fix $\Gam = \Gam_{2n+1,m,k}$ for some $n,m\in\omega$ and $k\in\omega\backslash\{0\}$. Let $b = b_{2n+1,m,k+1}$ and let $s = (u_0,...,u_{m-1})$ be the sequence of the first $m$ uniform indiscernibles. Towards a contradiction, suppose there exists a sequence $\langle A_\alpha : \alpha < b\rangle$ of distinct sets in $\Gam$.

Let $U\subset \R \times \R$ be a universal $\Gam$-set. Fix $x_0\in\R$ such that $U\in \Gamma_{2n+1,m,k}(x_0)$ and let $\psi_U$ be a formula realizing $U\in \Gamma_{2n+1,m,k}(x_0)$ (so $(x,y)\in U$ if and only if $M_{2n}[x,y,x_0]\models \psi_U(x,y,x_0,s)$). Let $B_\alpha = \{y: U_y = A_\alpha\}$.

Let $\I = \I(M_{2n+1})$. Let $\J = \{(P,\xi): P\in\I \wedge \xi < \delta_P\}$. Define the relation $\leq^{**}$ on $\J$ by $(P,\xi) \leq^{**} (Q,\zeta)$ if and only if $\pi_{P,R}(\xi)\leq \pi_{Q,R}(\zeta)$, where $R$ is a common iterate of $P$ and $Q$. Let $\leq^*\in\bf{\Delta^1_{2n+3}}$ be an initial segment of $\leq^{**}$ of length $b$.\footnote{This exists by the results of \cite{pwoim}.} The coding lemma gives $D^*\subset \R\times \R$, $D^*\in \bf{\Sigma^1_{2n+3}}$ such that
\begin{enumerate}
    \item $y^1\in dom(\leq^*) \implies D^*_{y^1} \neq \emptyset$ and
    \item $(y^1,y^2)\in D^* \implies y^2\in B_{|y^1|_{\leq^*}}$.    
\end{enumerate}
Let $D\subset \R\times\R\times\R$, $D\in \bf{\Pi^1_{2n+2}}$, be such that $(y^1,y^2)\in D^*$ if and only if $(\exists y^3)(y^1,y^2,y^3)\in D$. Replacing $x_0$ if necessary, we may assume $D\in \Pi^1_{2n+2}(x_0)$. Let $\psi_D$ be a formula such that $(y^1,y^2,y^3)\in D \longleftrightarrow M_{2n}[y^1,y^2,y^3,x_0] \models \psi_D(y^1,y^2,y^3,x_0)$.

Let $\I' = \I (M_{2n+1}(x_0))$.

\begin{definition}
    Say $M \in \mathcal{I}'$ is locally $\alpha$-stable if there is $\xi\in M$ such that $\pi_{\H^M,\infty}(\xi)=\alpha$. Write $\alpha_M$ for this ordinal $\xi$.
\end{definition}

\begin{definition}
    Say $M\in \mathcal{I}'$ is $\alpha$-stable if $M$ is locally $\alpha$-stable and whenever $N$ is a complete iterate of $M$, $\pi_{M,N}(\alpha_M) = \alpha_N$.
\end{definition}

\begin{lemma}
    For any $\alpha < b$, there exists an $\alpha$-stable $M\in \I'$.
\end{lemma}
\begin{proof}
    See Lemma 1.17 of \cite{hra}.
\end{proof}

\begin{definition}
    Suppose $M\in\I'$ is $\alpha$-stable and $p\in Ea_M$. Say $p$ is $\alpha$-good if $p$ forces the generic is a tuple $(ea^1,ea^2,ea^3)$ such that
    \begin{enumerate}
        \item \label{first part of alpha-good} $|ea^1|_{\leq^*} = \alpha$ and
        \item \label{second part of alpha-good} $M_{2n}[ea^1,ea^2,ea^3,x_0] \models \psi_D(ea^1,ea^2,ea^3,x_0)$.
    \end{enumerate}
\end{definition}

Being $\alpha$-good is first order. For part \ref{first part of alpha-good} of the definition, this follows from Fact \ref{direct limit approx definable}.

For $M\in \I'$, let $Th^M_k$ be the $\Sigma_k$-theory of $M$ with parameters in $M|\delta_M \cup s$. Let $\kappa_M$ be the least cardinal of $M$ such that $M\models$ ``$\kappa_M$ is $<\delta_M$-strong with respect to $Th^M_k$.'' Let $\kappa_\infty = \pi_{M,\infty}(\kappa_M)$ for some (equivalently any) $M\in \I'$.

\begin{lemma}
\label{least strong below b}
    $\kappa_\infty < b$
\end{lemma}
\begin{proof}
    The lemma is immediate from the fact that for $M\in \I'$, $\kappa_M$ is $\Sigma_{k+1}$-definable in $M$ from parameters in $\{x_0\} \cup s$.
\end{proof}

For $\alpha < b$ and $\alpha$-stable $M\in \I'$, let $S^M_\alpha$ be the set of conditions $q\in Ea_M$ such that there exists $r\in Ea_M$ satisfying
\begin{enumerate}[a)]
    \item $(q,r) \Vdash^M_{Ea_M\times Ea_M} ``M_{2n}[ea_l,ea_r^2,x_0] \models \psi_U(ea_l,ea_r^2,x_0,s)$,''
    \item $r$ is $\alpha$-good, and
    \item $q\in M|\kappa_M$.
\end{enumerate}

Let $S_\alpha = \pi_{M,\infty}(S^M_\alpha)$ for some (equivalently any) $\alpha$-stable $M\in \I'$. Note $S_\alpha$ is coded by an element of $P(\kappa_\infty)^{M'_\infty}$. But
\begin{align}
\label{powerset smaller than b}
    |P(\kappa_\infty)|^{M'_\infty} \leq (\kappa_\infty^+)^{M'_\infty} < \kappa^+_\infty \leq b
\end{align}

The second inequality holds because $(\kappa^+_\infty)^{M'_\infty}$ has cofinality $\omega$ since $\kappa_\infty$ is not measurable in $M'_\infty$. The last inequality holds by Theorem \ref{b is cardinal} and Lemma \ref{least strong below b}.

If $\alpha \neq \beta$ implies $S_\alpha \neq S_\beta$, then (\ref{powerset smaller than b}) contradicts that our sequence $\langle A_\alpha: \alpha < b\rangle$ had cardinality $b$. We will show the coding sets are distinct by reconstructing the original set $A_\alpha$ from its code $S_\alpha$.

Let $A'_\alpha$ be the set of $x\in\R$ such that $\exists M\in \I', q\in Ea_M$ satisfying
\begin{enumerate}
    \item $q\in S^M_\alpha$
    \item $x\models q$
    \item $x$ is $Ea_M$-generic over $M$
\end{enumerate}

It remains to show the following lemma:

\begin{lemma}
    $A'_\alpha = A_\alpha$
\end{lemma}
\begin{proof}
    First suppose $x\in A'_\alpha$. Fix $M$ and $q$ realizing $x\in A'_\alpha$ and $r$ realizing $q\in S^M_\alpha$. The choice of $r$ implies there exists a real $y$ which is $Ea_M$-generic over $M[x]$ such that $y\models r$.
    \begin{align*}
        (x,y)\models (q,r) &\implies M[x][y] \models ``M_{2n}[x,y^2,x_0]\models \psi_U(x,y^2,x_0,s)"\\
        &\implies M_{2n}[x,y^2,x_0]\models \psi_U[x,y^2,x_0,s]\\
        &\implies x\in U_{y^2}
    \end{align*}
    
    But $r$ is $\alpha$-good, so $U_{y^2} = A_\alpha$.\\

    Now suppose $x\in A_\alpha$. Fix $y = (y^1,y^2,y^3)$ such that $U_{y^2} = A_\alpha$ and $(y^1,y^2,y^3)\in D$. Let $\bar{M}\in\I'$ be $\alpha$-stable, $z\in \R$ code $\bar{M}$, $P = M_{2n+1}(x,y,z)$, and $S = StrLe[P,x_0]$. The next two claims are immediate from Facts \ref{generic for strle} and \ref{strle in directed system} above.

    \begin{claim}
        $x$ and $y$ are each $Ea_S$-generic over $S$.
    \end{claim}

    \begin{claim}
        $S$ is a complete iterate of $\bar{M}$
    \end{claim}

    \begin{claim}
        There are conditions $q,r\in Ea_S$ such that $x\models q$, $y\models r$, and $(q,r) \Vdash^S_{Ea_S\times Ea_S} ``M_{2n}[ea_l,ea_r^2,x_0]\models \psi_U(ea_l,ea_r^2,x_0,s)."$
    \end{claim}
    \begin{proof}
        $y$ satisfies some $\alpha$-good condition $r$. Let $y_0\in\R$ be $Ea_M$-generic over $S[x]$ such that $y_0\models r$. Then $y_0 = (y_0^1,y_0^2,y_0^3)$, where $(y_0^1,y_0^2,y_0^3)\in D$ and $|y^1_0|_{\leq_*} = \alpha$. This gives $U_{y_0^2} = A_\alpha$ from the definition of $D$. Thus $x\in U_{y_0^2}$.
        \begin{align}
            x\in U_{y_0^2} &\implies M_{2n}[x,y_0^2,x_0] \models \psi_U(x,y_0^2,x_0,s)\\
            &\label{U holds in extension of S by pair}\implies S[x][y_0]\models ``M_{2n}[x,y_0^2,x_0]\models \psi_U(x,y_0^2,x_0,s)"
        \end{align}

        Since (\ref{U holds in extension of S by pair}) holds for any $y_0$ which is $Ea_S$-generic over $S[x]$ and satisfies $r$, there is $q\in Ea_S$ such that $x\models q$ and $(q,r) \Vdash^S_{Ea_S\times Ea_S} ``M_{2n}[ea_l,ea_r^2, x_0]\models \psi_U(ea_l,ea_r^2,x_0,s)"$.
    \end{proof}

    \begin{claim}
           There are conditions $q,r\in Ea_S$ such that $q,r\in S|\kappa_S$, $x\models q$, $y\models r$, and $(q,r) \Vdash^S_{Ea_S\times Ea_S} ``M_{2n}[ea_l,ea_r^2,x_0]\models \psi_U(ea_l,ea_r^2,x_0,s)."$
    \end{claim}
    \begin{proof}
        The previous claim gives conditions $q',r'$ satisfying everything except $q',r'\in S|\kappa_S$. Let $T$ be the set of pairs of conditions $(\bar{q},\bar{r})\in Ea_{StrLe[P,x_0]}\times Ea_{StrLe[P,x_0]}$ such that
        \begin{align*}
            (\bar{q},\bar{r}) \Vdash^{StrLe[P,x_0]}_{Ea_{StrLe[P,x_0]}\times Ea_{StrLe[P,x_0]}} ``M_{2n}[ea_l,ea_r^2,x_0] \models \psi_U(ea_l,ea_r^2,x_0,s)."
        \end{align*}

        Recall $S = StrLe[P,x_0]$, so $(q',r')\in T$. Let $\kappa'$ be the least cardinal of $P$ which is $<\delta_P$-strong in $P$ with respect to $T$.

        \begin{subclaim}
            $\kappa'\leq\kappa_S$
        \end{subclaim}
        \begin{proof}
            Since $\delta_S = \delta_P$, it suffices to show $\kappa_S$ is $<\delta_P$-strong in $P$ with respect to $T$. Note             
            \begin{align}
                \label{T is section of Th} T = \{a: (j,a)\in Th^S_k\} \text{ for some } j\in \omega.
            \end{align}

            Let $\zeta < \delta_S$ and let $E$ be an extender in $S$ such that $E$ realizes $\kappa_S$ is $\zeta$-strong with respect to $Th^S_k$. Let $E^*$ be the extender in $P$ which backgrounds $E$. Then
            \begin{align*}
                i_{E^*}(T) \cap P|\zeta &= i_E(T) \cap S|\zeta\\
                &= T\cap S|\zeta\\
                &= T\cap P|\zeta.
            \end{align*}

            The first and third equalities hold because $T\subset S$, the second because $i_E$ moves $T$ correctly by (\ref{T is section of Th}). Since $\zeta$ was arbitrary, $\kappa'\leq \kappa_S$.
        \end{proof}

        Fix $\zeta < \delta_S$ such that $q',r'\in P|\zeta$ and an extender $E\in P$ such that $crit(E) = \kappa'$ and $E$ is $\zeta$-strong with respect to $T$. Then
        \begin{align*}
            i_E(T\cap S|\kappa') \supset T \cap S|\zeta \supset \{(q',r')\}.
        \end{align*}

        In particular,
        \begin{align*}
            Ult(P,E) \models ``\exists q,r\in i_E(T\cap S|\kappa') \text{ s.t. } x\models q \wedge y \models r."
        \end{align*}

        Then by elementarity,
        \begin{align*}
            P \models ``\exists q,r\in T\cap S|\kappa' \text{ s.t. } x\models q \wedge y \models r."
        \end{align*}            
    \end{proof}

    It remains to show the condition $r$ from the previous claim can be extended to an $\alpha$-good condition. It suffices to show the following claim:

    \begin{claim}
        There is an $\alpha$-good condition $p\in Ea_S$ such that $y\models p$.
    \end{claim}
    \begin{proof}
        Recall $y$ is a triple $(y^1,y^2,y^3)$ such that $|y^1|_{\leq_*} = \alpha$ and $(y^1,y^2,y^3)\in D$. In particular, $|y^1|_{\leq_*} = \alpha$ and $M_{2n}(y^1,y^2,y^3,x_0) \models \psi_D(y^1,y^2,y^3,x_0)$. Both are first order in $S[y]$, so there is $p\in Ea_S$ forcing these such that $y\models p$. $p$ is the desired $\alpha$-good condition.
    \end{proof}
\end{proof}

\section{Open Problems}

Many interesting questions about unreachability remain, even for pointclasses between the levels of the projective hierarchy. Harrington proved there is no uncountable sequence of distinct $\bf{\Gamma}$-sets for any $\bf{\Gamma} < \bf{\Delta^1_1}$ (see \cite{Harrington}). We are hardly the first to suspect the same may hold for $\bf{\Delta^1_3}$:

\begin{conjecture}
    There is no pointclass $\bf{\Gamma} <_w \bf{\Delta^1_3}$ such that $\aleph_{\omega+1}$ is $\bf{\Gamma}$-reachable.
\end{conjecture}

The analogous question for $\bf{\Delta_{2n+3}}$ is also open. The computation of $b_{2n+1,m}$ for $n\geq 1$ and $m\in\omega$ remains an important open problem in inner model theory. Sargsyan conjectured in \cite{on_the_pwo} that $b_{2n+1,0} = \bf{\delta^1_{2n+2}}$ for all $n$. But even computing the values of $b_{2n+1,m}$ would still leave other questions about unreachability for pointclasses between $\bf{\Sigma^1_{2n+2}}$ and $\bf{\Delta^1_{2n+3}}$. Theorem \ref{main thm} implies there are exactly $\omega$ cardinals $\lambda$ between $\bf{\delta^1_4}$ and $\bf{\delta^1_6}$ such that $\lambda$ is the least $\bf{\Gamma^3_m}$-unreachable for some $m$.

\begin{conjecture}
    If $\lambda$ is a cardinal between $\bf{\delta^1_4}$ and $\bf{\delta^1_6}$ and $cof(\lambda) >\omega$, then there is a pointclass $\bf{\Gamma}$ such that $\lambda$ is the least $\bf{\Gamma}$-unreachable.
\end{conjecture}

\subsection*{Acknowledgements}

The author would like to thank Itay Neeman and Grigor Sargsyan for many helpful conversations. Thanks also to several participants in the 2nd Irvine Conference on Inner Model Theory for raising questions which led to these results.

\printbibliography

\end{document}